\documentclass{article}
\usepackage{amsfonts}
\usepackage{graphicx}
\usepackage[mathlines]{lineno}
\usepackage{amsmath}
\usepackage{amssymb}
\usepackage{dsfont}
\usepackage[colorlinks=true]{hyperref}

\makeatletter
\renewcommand*{\@fnsymbol}[1]
{\ensuremath{\ifcase#1\or 
*\or \mathsection \or
\ddagger \else\@ctrerr\fi}}
\makeatother
\hypersetup{urlcolor=blue, citecolor=red}
\newtheorem{theorem}{Theorem}

\newtheorem{corollary}[theorem]{Corollary}

\newtheorem{definition}[theorem]{Definition}
\newtheorem{example}[theorem]{Example}

\newtheorem{lemma}[theorem]{Lemma}

\newenvironment{proof}[1][Proof]{\noindent\textbf{#1.} }{\ \rule{0.5em}{0.5em}}
\newcommand*\patchAmsMathEnvironmentForLineno[1]{  \expandafter \let \csname old#1\expandafter \endcsname \csname
#1\endcsname
  \expandafter \let \csname oldend#1\expandafter \endcsname \csname
end#1\endcsname
  \renewenvironment{#1}     {\linenomath \csname old#1\endcsname}     {\csname oldend#1\endcsname \endlinenomath}}
\newcommand*\patchBothAmsMathEnvironmentsForLineno[1]{  \patchAmsMathEnvironmentForLineno{#1}  \patchAmsMathEnvironmentForLineno{#1*}}
\patchBothAmsMathEnvironmentsForLineno{equation}
\patchBothAmsMathEnvironmentsForLineno{align}
\patchBothAmsMathEnvironmentsForLineno{flalign}
\patchBothAmsMathEnvironmentsForLineno{alignat}
\patchBothAmsMathEnvironmentsForLineno{gather}
\patchBothAmsMathEnvironmentsForLineno{multline}

\begin{document}

\title{\vspace{-1in}
\parbox{\linewidth}{\footnotesize\noindent
\textbf{}  \hfill 
\newline
} 
\vspace{0pt} \\
Sharp estimates for the unique solution of the Hadamard-type two-point
fractional boundary value problems}
\date{{ Submitted on 24 April 2020; accepted for publication in \textquotedblleft Applied Mathematics E-Notes\textquotedblright}}

\author{Zaid Laadjal\thanks{Department of Mathematics and Computer Sciences, ICOSI Laboratory, Abbes Laghrour University, Khenchela, 40004, Algeria. Email: zaid.laadjal@yahoo.com}\ , Adjeroud Nacer\thanks{Department of Mathematics and Computer Sciences, ICOSI Laboratory, Abbes Laghrour University, Khenchela, 40004, Algeria. Email: adjeroud\_n@yahoo.fr} }
\maketitle

\begin{abstract}
In this short note, we present the sharp estimate for the existence of a
unique solution for a Hadamard-type fractional differential equations with
two-point boundary value conditions. The method of analysis is obtained by
using the integral of the Green's function and the Banach contraction
principle. Further, we will also obtain a sharper lower bound of the
eigenvalues for an eigenvalue problem. Two examples are presented to clarify
the applicability of the essential results.
\end{abstract}

\section{Introduction and preliminaries}

In the book \cite{Kelley} Kelley and Peterson considered the following
classical two-point boundary value problems: 
\begin{equation}
\left\{ 
\begin{array}{ll}
u^{\prime \prime }(x)={\mathcal{F}}(x,u(x)), & a<x<b, \\ 
u(a)=A,\text{ }u(b)=B, & A,B\in \mathds{R},
\end{array}
\right.   \label{1}
\end{equation}%
where $a,b\in \mathds{R},$\ and they included the following result:

\begin{theorem}[\protect\cite{Kelley}, Theorem 7.7]
Let ${\mathcal{F}}:[a,b]\times \mathds{R}\rightarrow \mathds{R}$ be a continuous function satisfying the assumption:

\begin{itemize}
\item[$(H_{1})$] There exists $K>0$ such that $\left\vert {\mathcal{F}}%
(x,\omega )-{\mathcal{F}}(x,\varpi )\right\vert \leq K\left\vert \omega
-\varpi \right\vert $ for all $(x,\omega ),(x,\varpi )\in \lbrack a,b]\times \mathds{R}.$
\end{itemize}

Then the boundary value problem (\ref{1}) has a unique solution on $[a,b]$
if $b-a<2\sqrt{2/K}.$
\end{theorem}

Ferreira in 2016 \cite{Ferreira-1} discussed the existence and uniqueness of
solutions for the following fractional boundary value problems with
Reimman-Liouville fractional derivative: 
\begin{equation}
\left\{ 
\begin{array}{ll}
^\mathcal{R}\mathfrak{D}_{a}^{\sigma }u(x)=-{\mathcal{F}}(x,u(x)), & 
a<x<b,1<\sigma \leq 2, \\ 
u(a)=0,\text{\ }u(b)=B, & B\in \mathds{R},
\end{array}
\right.  \label{2}
\end{equation}
so he included this result given by:

\begin{theorem}[\protect\cite{Ferreira-1}]
Let ${\mathcal{F}}:[a,b]\times \mathds{R}\rightarrow \mathds{R}$ be a continuous function satisfying the assumption $(H_{1}).$ Then the
boundary value problem (\ref{2}) has a unique solution on $[a,b]$ if $
b-a<(\sigma ^{(\sigma +1)/\sigma }\Gamma ^{1/\sigma }(\sigma ))/(K^{1/\sigma
}(\sigma -1)^{(\sigma -1)/\sigma }).$
\end{theorem}

In 2019, Ferreira \cite{Ferreira-2} corrected a recent uniqueness result 
\cite{Bashir} for a two-point fractional boundary value problem with Caputo
derivative:
\begin{equation}
\left\{ 
\begin{array}{ll}
^{\mathcal{C}}{\mathfrak{D}}_{a}^{\sigma }u(x)=-{\mathcal{F}}(x,u(x)), & 
a<x<b,1<\sigma \leq 2, \\ 
u(a)=A,\text{\ }u(b)=B, & A,B\in \mathds{R},
\end{array}%
\right.   \label{3}
\end{equation}
and he came to the following theorem:

\begin{theorem}[\protect\cite{Ferreira-2}]
Let ${\mathcal{F}}:[a,b]\times \mathds{R}\rightarrow \mathds{R}$ be a continuous function satisfying the assumption $(H_{1}).$ Then the
boundary value problem (\ref{3}) has a unique solution on $[a,b]$ if $
M(\sigma ,a,b)<1/K,$ where 
\begin{equation*}
M(\sigma ,a,b)=\frac{1}{\Gamma (\sigma +1)}\underset{x\in \lbrack a,b]}{\max 
}\Big(-2(x-\varphi (x))^{\sigma }+2\frac{(x-a)(b-\varphi (x))^{\sigma }}{b-a}
+(x-a)^{\sigma }-(x-a)(b-a)^{\sigma -1}\Big),
\end{equation*}
with $\varphi (x)=\left( (\frac{x-a}{b-a})^{\frac{1}{\sigma -1}}b-1\right)
/\left( (\frac{x-a}{b-a})^{\frac{1}{\sigma -1}}-1\right) .$
\end{theorem}

See \cite{Ferreira-1,Ferreira-2} and references therein for more details.

In the last few decades, the differential equations involving a fractional
order have witnessed a wide attention from the researchers, because they
were extensively implemented in daily life and in various scientific and
technological fields and in many branches including physics, biology,
chemistry, economics, astronomy, control theory, viscoelastic materials,
robotics, signal processing, electromagnetism, electrodynamics of complex
medium, anomalous diffusion and fractured media, electromagnetism, potential
theory and electro statistics, polymer rheology, and aerodynamics, ... etc,
we refer the interested reader to paper \cite{bib6}, and the references contained therein.

It is well known that the existence of solution plays a main important role
in the theory and applications of fractional differential equations with
boundary conditions. Recently, many researchers are interested in studying
the Hadamard-type fractional boundary value problems, where there are
several results about the existence of solutions for the differential
equations with Hadamard derivative, we refer the reader to the book \cite
{Book.Hadamard} that contains the most important works that have been
published in this domain. In addition, some researchers are interested in studying the stability of solutions to fractional differential equations, including Laypunov stability, exponential stability, Mittag-Leffler stability, and Hyers–Ulam stability, have been introduced. Among these concepts, Hyers–Ulam stability analysis was recognized as a simple method of investigation. We refer the readers to \cite{1,2,3,4,5,6},  and the references contained therein.

Motivated by the above mentioned works and the papers \cite{Zaid,Ma}, in
this paper, we investigated the sharp estimate for the unique solution of
the following fractional differential equation with Hadamard derivative: 
\begin{equation}
\left\{ 
\begin{array}{ll}
^{\mathcal{H}}{\mathfrak{D}}_{a}^{\sigma }u(x)=-{\mathcal{F}}(x,u(x)), & 
0<a<x<b,1<\sigma \leq 2, \\ 
u(a)=0,\text{ \ }u(b)=B, & B\in \mathds{R},
\end{array}
\right.  \label{7}
\end{equation}
where ${\mathcal{F}}$ is a given function, $^{\mathcal{H}}{\mathfrak{D}}
_{a}^{\sigma }$ denotes the Hadamard fractional derivative of order $\sigma $
, and $B$ is real constant. Further, we will also obtain a sharp estimate
for lower bound for the eigenvalues for the follwing eigenvalue problem
\begin{equation}
\left\{ 
\begin{array}{ll}
^{\mathcal{H}}{\mathfrak{D}}_{a}^{\sigma }u(x)=\lambda u(x), & 
0<a<x<b,1<\sigma \leq 2, \\ 
u(a)=0=u(b). & 
\end{array}
\right.  \label{v-p}
\end{equation}

We start now to present some fundamental definitions and lemmas which will
be used in this work.

\begin{definition}[\protect\cite{Book.Hadamard,Kilbas}]
\label{dif1}Let $0<a\leq b$ and $\sigma \in \mathds{R}^{+}$ where$\
n-1<\sigma \leq n$ with $n\in\mathds{N}.$ The Hadamard fractional integral of ordre $\sigma $ for a function $g$ is
defined by: $^{\mathcal{H}}{\mathcal{I}}_{a}^{0}g(x)=g(x)$ and 
\begin{equation}
^{\mathcal{H}}{\mathcal{I}}_{a}^{\sigma }g(x)=\frac{1}{\Gamma (\sigma )}
\int_{a}^{x}\left( \ln \frac{x}{\tau }\right) ^{\sigma -1}g(\tau )\frac{
d\tau }{\tau }\text{\ for }\sigma >0.  \label{8}
\end{equation}
\end{definition}

\begin{definition}[\protect\cite{Book.Hadamard,Kilbas}]
\label{dif2} {\small \ } Let $0<a<b;$ $\delta =x\frac{d}{dx}$ and let $
AC[a,b]$ be the space of functions $g$ which are absolutely continuous on $
[a,b],$ and $AC_{\delta }^{n}[a,b]=\{g:[a,b]\times \mathds{R}\rightarrow \mathds{R}$ s.t. $\delta ^{n-1}[g(x)]\in AC[a,b]\}.$ The Hadamard fractional
derivative of order $\sigma \geq 0$ for a function $g\in AC_{\delta
}^{n}[a,b]$ is defined by: $^{\mathcal{H}}{\mathfrak{D}}_{a}^{0}g(x)=g(x),$
and 
\begin{equation}
^{\mathcal{H}}{\mathfrak{D}}_{a}^{\sigma }g(x)=\frac{1}{\Gamma (n-\sigma )}
\left( x\frac{d}{dx}\right) ^{n}\int_{a}^{x}\left( \ln \frac{x}{\tau }
\right) ^{n-\sigma -1}g(\tau )\frac{d\tau }{\tau }\text{\ for }\sigma >0,
\label{9}
\end{equation}
where\ $n-1<\sigma \leq n,$ $n\in \mathds{N}.$
\end{definition}

\begin{lemma}[\protect\cite{Book.Hadamard,Kilbas}]
\label{l1} Let\ $0<a\leq b$, and $\sigma >0$\ where\ $n-1<\sigma \leq n$, $
n\in \mathds{N}$.\ The differential equation\ $^{\mathcal{H}}{\mathfrak{D}}_{a}^{\sigma
}u(x)=0$\ has the general solution: 
\begin{equation}
u(x)=\sum_{i=1}^{i=n}c_{i}\left( \ln \frac{x}{a}\right) ^{\sigma -i},\;x\in
\lbrack a,b],  \label{10}
\end{equation}%
where $c_{i}\in \mathds{R}$ $(i=1,...,n)$ are arbitrary constants. And
moreover 
\begin{equation}
^{\mathcal{H}}{\mathcal{I}}_{a}^{\sigma }{}^{\mathcal{H}}{\mathfrak{D}}
_{a}^{\sigma }u(x)=u(x)+\sum_{i=1}^{i=n}c_{i}\left( \ln \frac{x}{a}\right)
^{\sigma -i}.  \label{11}
\end{equation}
\end{lemma}

\begin{lemma}
\label{l3} Let $y\in C([a,b],\mathds{R})\cap L^{1}([a,b],\mathds{R})$,\ the solution of following linear fractional boundary value problem 
\begin{equation}
\left\{ 
\begin{array}{ll}
^{\mathcal{H}}{\mathfrak{D}}_{a}^{\sigma }u(x)=-y(x), & 0<a<x<b,1<\sigma
\leq 2, \\ 
u(a)=0,\text{ }u(b)=B, & B\in \mathds{R},
\end{array}
\right.   \label{13}
\end{equation}
is given by 
\begin{equation*}
u(x)=\int_{a}^{b}G(x,\tau )y(\tau )d\tau +B\left( \frac{\ln \frac{x}{a}}{\ln 
\frac{b}{a}}\right) ^{\sigma -1},
\end{equation*}
where 
\begin{equation}
G(x,\tau )=\frac{1}{\Gamma (\sigma )}\left\{ 
\begin{array}{ll}
\left( \frac{\ln \frac{x}{a}}{\ln \frac{b}{a}}\right) ^{\sigma -1}\left( \ln 
\frac{b}{\tau }\right) ^{\sigma -1}\frac{1}{\tau }-\left( \ln \frac{x}{\tau}
\right) ^{\sigma -1}\frac{1}{\tau }, & a\leq \tau \leq x\leq b, \\ 
&  \\ 
\left( \frac{\ln \frac{x}{a}}{\ln \frac{b}{a}}\right) ^{\sigma -1}\left( \ln 
\frac{b}{\tau }\right) ^{\sigma -1}\frac{1}{\tau }, & a\leq x\leq \tau \leq
b.
\end{array}
\right.   \label{15}
\end{equation}
\end{lemma}

\begin{proof}
Applying the operator $^{\mathcal{H}}{\mathcal{I}}_{a}^{\sigma }$ on the equation $^{
\mathcal{H}}{\mathfrak{D}}_{a}^{\sigma }u(x)=-y(x)$, we get 
\begin{equation}
u(x)=-\dfrac{1}{\Gamma (\sigma )}\int_{a}^{x}\left( \ln \frac{x}{\tau }
\right) ^{\sigma -1}y(\tau )\,\frac{d\tau }{\tau }+c_{1}\left( \ln \frac{x}{a}\right) ^{\sigma -1}+c_{2}\left( \ln \frac{x}{a}\right) ^{\sigma -2},
\label{16}
\end{equation}
where $c_{1},c_{2}\in \mathds{R}$.\newline
Using the boundary conditions $u(a)=0$ and $u(b)=B,$ we get $c_{2}=0$ and 
\begin{equation*}
c_{1}=\dfrac{1}{\Gamma (\sigma )}\left( \ln \frac{b}{a}\right) ^{1-\sigma
}\int_{a}^{b}\left( \ln \frac{b}{\tau }\right) ^{\sigma -1}y(\tau )\frac{d\tau }{\tau }+B\left( \ln \frac{b}{a}\right) ^{1-\sigma }.
\end{equation*}
Substituting the values of $c_{1}$ and $c_{2}$ in (\ref{16}), we obtain 
\begin{align*}
u(x)& =\frac{1}{\Gamma (\sigma )}\left( \frac{\ln \frac{x}{a}}{\ln \frac{b}{a}}\right) ^{\sigma -1}\int_{a}^{b}\left( \ln \frac{b}{\tau }\right) ^{\sigma
-1}y(\tau )\frac{d\tau }{\tau }-\frac{1}{\Gamma (\sigma )}\int_{a}^{x}\left(
\ln \frac{x}{\tau }\right) ^{\sigma -1}y(\tau )\frac{d\tau }{\tau }+B\left( 
\frac{\ln \frac{x}{a}}{\ln \frac{b}{a}}\right) ^{\sigma -1} \\
& =\frac{1}{\Gamma (\sigma )}\int_{a}^{x}\left[ \left( \frac{\ln \frac{x}{a}
}{\ln \frac{b}{a}}\right) ^{\sigma -1}\left( \ln \frac{b}{\tau }\right)
^{\sigma -1}-\left( \ln \frac{x}{\tau }\right) ^{\sigma -1}\right] y(\tau )
\frac{d\tau }{\tau }+\frac{1}{\Gamma (\sigma )}\int_{x}^{b}\left( \frac{\ln 
\frac{x}{a}}{\ln \frac{b}{a}}\right) ^{\sigma -1} \\
& \times \left( \ln \frac{b}{\tau }\right) ^{\sigma -1}y(\tau )\frac{d\tau }{
\tau }+B\left( \frac{\ln \frac{x}{a}}{\ln \frac{b}{a}}\right) ^{\sigma -1} \\
& =\int_{a}^{b}G(x,\tau )\;d\tau +B\left( \frac{\ln \frac{x}{a}}{\ln \frac{b
}{a}}\right) ^{\sigma -1}.
\end{align*}
Hence, the proof is completed.
\end{proof}

\section{Main results}

This section is devoted to prove the main results of the problem (\ref{7}),
and present a lower bound of the eigenvalues for the eigenvalue problem (\ref
{v-p}).

\begin{lemma}
\label{l4} The Green's function $G$ defined in Lemma \ref{l3} has the
following property: 
\begin{equation}
\underset{x\in \lbrack a,b]}{\max }\int_{a}^{b}\left\vert G\left( x,\tau
\right) \right\vert d\tau =\frac{\left( \sigma -1\right) ^{\sigma -1}\left(
\ln \frac{b}{a}\right) ^{\sigma }}{\sigma ^{\sigma +1}\Gamma (\sigma )}.
\label{max.int}
\end{equation}
\end{lemma}

\begin{proof}
From Lemma$\ $4 of \cite{Zaid}, we have $G(x,\tau )\geq 0$\ for all $(x,\tau
)\in \lbrack a,b]\times \lbrack a,b].$ Therefore,
\begin{eqnarray*}
\Gamma (\sigma )\int_{a}^{b}|G\left( x,\tau \right) |d\tau  &=&\int_{a}^{x}
\left[ \left( \frac{\ln \frac{x}{a}}{\ln \frac{b}{a}}\right) ^{\sigma
-1}\left( \ln \frac{b}{\tau }\right) ^{\sigma -1}-\left( \ln \frac{x}{\tau }
\right) ^{\sigma -1}\right] \frac{d\tau }{\tau } \\
&&+\int_{x}^{b}\left( \frac{\ln \frac{x}{a}}{\ln \frac{b}{a}}\right)
^{\sigma -1}\left( \ln \frac{b}{\tau }\right) ^{\sigma -1}\frac{d\tau }{\tau 
} \\
&=&\left( \frac{\ln \frac{x}{a}}{\ln \frac{b}{a}}\right) ^{\sigma
-1}\int_{a}^{b}\left( \ln \frac{b}{\tau }\right) ^{\sigma -1}\frac{d\tau }{
\tau }-\int_{a}^{x}\left( \ln \frac{x}{\tau }\right) ^{\sigma -1}\frac{d\tau 
}{\tau } \\
&=&\frac{1}{\sigma }\left( \ln \frac{b}{a}\right) ^{1-\sigma }\left( \ln 
\frac{x}{a}\right) ^{\sigma -1}\left( \ln \frac{b}{a}\right) ^{\sigma }-
\frac{1}{\sigma }\left( \ln \frac{x}{a}\right) ^{\sigma },
\end{eqnarray*}
which yields 
\begin{equation}
\Gamma (\sigma +1)\int_{a}^{b}G\left( x,\tau \right) d\tau =\left( \ln \frac{b}{a}\right) \left( \ln \frac{x}{a}\right) ^{\sigma -1}-\left( \ln \frac{x}{a}\right) ^{\sigma }.  \label{31a}
\end{equation}%
It follows that we need to get the maximum value of the function
\begin{equation}
g(x)=\left( \ln \frac{b}{a}\right) \left( \ln \frac{x}{a}\right) ^{\sigma
-1}-\left( \ln \frac{x}{a}\right) ^{\sigma },\ \ x\in \lbrack a,b].
\label{202}
\end{equation}%
Observe that $g(x)\geq 0$ for all $x\in \lbrack a,b],$ and $g(a)=g(b)=0.$ 
\newline
Now we differentiate $g(x)\ $\ on $(a,b)$ to get 
\begin{equation*}
g^{\prime }(x)=\frac{\left( \sigma -1\right) }{x}\left( \ln \frac{b}{a}%
\right) \left( \ln \frac{x}{a}\right) ^{\sigma -2}-\frac{\sigma }{x}\left(
\ln \frac{x}{a}\right) ^{\sigma -1},
\end{equation*}%
from which follows that $g^{\prime }(x^{\ast })=0$ has a unique zero,
attained at the point 
\begin{equation*}
x^{\ast }=a\left( \frac{b}{a}\right) ^{\left( \sigma -1\right) /\sigma }.
\end{equation*}%
It is easily seen that $x^{\ast }\in (a,b).$\ Because $g(x)$ is continuous
function and $x^{\ast }\in (a,b)$, we conclude that 
\begin{eqnarray}
\underset{x\in \lbrack a,b]}{\max }g(x) &=&g(x^{\ast })  \notag \\
&=&\left( \ln \frac{b}{a}\right) \left( \ln \left( \frac{b}{a}\right)
^{\left( \sigma -1\right) /\sigma }\right) ^{\sigma -1}-\left( \ln \left( 
\frac{b}{a}\right) ^{\left( \sigma -1\right) /\sigma }\right) ^{\sigma } 
\notag \\
&=&\frac{1}{\sigma -1}\left( \frac{\sigma -1}{\sigma }\ln \frac{b}{a}\right)
^{\sigma }  \notag \\
&=&\frac{\left( \sigma -1\right) ^{\sigma -1}\left( \ln \frac{b}{a}\right)
^{\sigma }}{\sigma ^{\sigma }}.  \label{max.g}
\end{eqnarray}%
By (\ref{31a}), (\ref{202}) and (\ref{max.g}) we get the formula (\ref%
{max.int}). The proof is completed.$\ $
\end{proof}

\begin{theorem}
\label{Th.4} \label{Th1} Let ${\mathcal{F}}:[a,b]\times \mathds{R}\rightarrow \mathds{R}$ be a continuous function satisfying the assumption $(H_{1}).$ If 
\begin{equation}
\frac{b}{a}<\exp \left( \frac{\sigma ^{(\sigma +1)/\sigma }\Gamma ^{1/\sigma
}(\sigma )}{\left( \sigma -1\right) ^{(\sigma +1)/\sigma }K^{1/\sigma }}%
\right) .  \label{201}
\end{equation}%
Then the Fractional boundary value problem (\ref{7}) has a unique solution
on $[a,b].$

\begin{proof}
Let $E=C\left( [a,b],\mathds{R}\right) $ be the Banach space endowed with the norm $\left\Vert u\right\Vert
=\sup_{x\in \lbrack a,b]}\left\vert u(x)\right\vert$ (see Proposition 2.18 in \cite{chapter2}), and we define the
operator $\mathfrak{R}:E\rightarrow E$ $\ $by 
\begin{equation*}
\mathfrak{R}u(x)=\int_{a}^{b}G(x,\tau )u(\tau )d\tau +B\left( \frac{\ln 
\frac{x}{a}}{\ln \frac{b}{a}}\right) ^{\sigma -1},
\end{equation*}%
where the function $G$ is given by (\ref{15}).\newline
Notice that the prolem (\ref{7}) has a solution $u$ if only if $u$ is fixed
point of the operator $\mathfrak{R}$.\newline
For all $(x,u),(x,v )\in \lbrack a,b]\times E,$ we have 
\begin{eqnarray*}
\left\vert \mathfrak{R}u\left( x\right) -\mathfrak{R}v\left( x\right)
\right\vert &\leq &\int_{a}^{b}G(x,\tau )\left\vert {\mathcal{F}}(\tau
,u(\tau ))-{\mathcal{F}}(\tau ,v(\tau ))\right\vert d\tau \\
&\leq &\int_{a}^{b}KG(x,\tau )\left\vert u\left( \tau \right) -v\left( \tau
\right) \right\vert d\tau \\
&\leq &K\int_{a}^{b}G\left( x,\tau \right) d\tau \left\Vert u-v\right\Vert ,
\end{eqnarray*}%
using the formula (\ref{max.int}) yields\ 
\begin{equation*}
\left\Vert \mathfrak{R}u-\mathfrak{R}v\right\Vert \leq \frac{K\left( \sigma
-1\right) ^{\sigma -1}\left( \ln \frac{b}{a}\right) ^{\sigma }}{\sigma
^{\sigma +1}\Gamma (\sigma )}\left\Vert u-v\right\Vert .
\end{equation*}%
Can be easily check that the assumption (\ref{201}) leads to principle of
contraction mapping. Hence, the operator $\mathfrak{R}$\ is contraction
mapping, we conclude that the problem (\ref{7}) has a unique solution.
\end{proof}
\end{theorem}

Now we present a lower bound of the eigenvalues for the eigenvalue problem (%
\ref{v-p}).

\begin{theorem}
\label{Th2} If the eigenvalue problem (\ref{v-p}) has a non-trivial
continuous solution, then 
\begin{equation}
|\lambda |\geq \frac{\sigma ^{\sigma +1}\Gamma (\sigma )}{\left( \sigma
-1\right) ^{\sigma -1}\left( \ln \frac{b}{a}\right) ^{\sigma }},  \label{52}
\end{equation}
\end{theorem}

\begin{proof}
From Lemma \ref{l3}, the solution of the problem (\ref{v-p}) can be written
as follows%
\begin{equation*}
u\left( x\right) =\int_{a}^{b}\lambda G\left( x,\tau \right) u\left( \tau
\right) d\tau .
\end{equation*}

which yields%
\begin{equation*}
\left\Vert u\right\Vert \leq |\lambda |\left\Vert u\right\Vert \underset{%
x\in \lbrack a,b]}{\max }\int_{a}^{b}\left\vert G\left( x,\tau \right)
\right\vert d\tau
\end{equation*}%
Since $u$ is non-trivial, then $\left\Vert u\right\Vert \neq 0.$ So, using
now to the formula of the Green function $G$ proved in Lemma \ref{l4}, we
get 
\begin{equation*}
1\leq |\lambda |\underset{x\in \lbrack a,b]}{\max }\int_{a}^{b}\left\vert
G\left( x,\tau \right) \right\vert d\tau =|\lambda |\frac{\left( \sigma
-1\right) ^{\sigma -1}\left( \ln \frac{b}{a}\right) ^{\sigma }}{\sigma
^{\sigma +1}\Gamma (\sigma )},
\end{equation*}%
from which the inequality (\ref{52}) follows. The proof is completed.
\end{proof}

\bigskip

We have the following result about the nonexistence for solutions of the
boundary value problem (\ref{v-p}).

\begin{corollary}
\label{Cor1}If%
\begin{equation}
|\lambda |<\frac{\sigma ^{\sigma +1}\Gamma (\sigma )}{\left( \sigma
-1\right) ^{\sigma -1}\left( \ln \frac{b}{a}\right) ^{\sigma }},  \label{53}
\end{equation}%
Then the boundary value\ problem (\ref{v-p}) has no non-trivial solution.
\end{corollary}

\begin{proof}
Assume the contrary. Then the boundary value problem (\ref{v-p}) has a
non-trivial solution $u.$ By Theorem \ref{Th2}, inequality (\ref{52}) holds.
This contradicts assumption (\ref{53}). The proof is completed.
\end{proof}

\begin{example}
We consider the following Hadamard frational\ boundary value problem%
\begin{equation}
\left\{ 
\begin{array}{l}
^{\mathcal{H}}{\mathfrak{D}}_{a}^{{3/2}}u(x)=(x-1)^{2}+\sqrt{x-1+u^{2}(x)},%
\text{ }1<x<e, \\ 
u(1)=0,\text{ }u(e)=1,%
\end{array}%
\right.   \label{401}
\end{equation}%
where $e$ is an irrational number and it's defined by the infinite series $e=\sum_{k=0}^{+\infty }\frac{1}{k!}$ and approximately equal to $2.718281828459.
$\newline
Here $\sigma =\frac{3}{2}$\ and ${\mathcal{F}}(x,u)=(x-1)^{2}+\sqrt{x-1+u^{2}%
}.$ For all $(x,u)\in (1,e]\times \mathds{R},$\ we have 
\begin{eqnarray*}
\left\vert \partial _{u}{\mathcal{F}}(x,u)\right\vert  &=&\frac{|u|}{\sqrt{%
x-1+u^{2}}} \\
&\leq &1.
\end{eqnarray*}%
Choose $K=1.$ So, by using the given values we get%
\begin{equation*}
\exp \left( \frac{\sigma ^{(\sigma +1)/\sigma }\Gamma ^{1/\sigma }(\sigma )}{%
\left( \sigma -1\right) ^{(\sigma +1)/\sigma }K^{1/\sigma }}\right) =\exp
\left( \frac{3}{4}\left( 9\pi \right) ^{1/3}\right) >e.
\end{equation*}%
Then the inequality (\ref{201}) is satisfied. Hence, by Theorem \ref{Th1},
we conclude that the Hadamard fractional boundary value problem (\ref{401})
has a unique solution on the interval $[1,e].$
\end{example}

\begin{example}
Consider the following eigenvalue problem%
\begin{equation}
\left\{ 
\begin{array}{l}
^{\mathcal{H}}{\mathfrak{D}}_{a}^{{3/2}}u(x)=\lambda u(x),\text{ }1<x<e, \\ 
u(1)=0=u(e),%
\end{array}%
\right.  \label{Pb.ex2}
\end{equation}%
Here $\sigma =\frac{3}{2},$ and$\ [a,b]=[1,e]${\small $.$ }So, we obtain%
\begin{equation}
\frac{\sigma ^{\sigma +1}\Gamma (\sigma )}{\left( \sigma -1\right) ^{\sigma
-1}\left( \ln \frac{b}{a}\right) ^{\sigma }}=\frac{9\sqrt{3\pi }}{8},
\end{equation}%
By Theorem \ref{Th2}, we conclude that: If $\lambda $ is an eigenvalue of
the problem (\ref{Pb.ex2}), we must have $\left\vert \lambda \right\vert
\geq 9\sqrt{3\pi }/8.$
\end{example}

\smallskip 

\textbf{Acknowledgment.} The authors would like to thank the anonymous referees for their useful remarks that led that improved the paper.

\end{document}